\documentclass[10pt,a4paper]{amsart}
\usepackage{cite}
\usepackage[utf8]{inputenc}
\usepackage{amsmath}
\usepackage{amsfonts}
\usepackage{amssymb}
\usepackage{amsthm}
\usepackage{enumitem}
\usepackage[all, cmtip]{xy}
\usepackage{verbatim}
\usepackage{caption}
\usepackage{comment}

\DeclareMathOperator{\join}{Join}
\newtheorem{theorem}{Theorem}
\newtheorem{proposition}{Proposition}
\newtheorem{corollary}{Corollary}
\newtheorem{lemma}{Lemma}
\newtheorem{remark}{Remark}
\newtheorem{example}{Example}
\newtheorem{definition}{Definition}

\begin{document}
\title[Compactification of a diagonal action]{Compactification of a diagonal action on the product of $\text{CAT}(-1)$ spaces}
\author{Teresa García}
\address{Departament de Matemàtiques,
Universitat Autònoma de Barcelona,
E-08193 Bellaterra, Spain} \email{tggalvez@mat.uab.cat}
\thanks{This work has been supported by grants BES-2013-06571 (MINECO) and MTM2015-66165-P (MINECO/FEDER). The author thanks her advisor, Joan Porti, for his unvaluable advice.}
\begin{abstract}
Let $X$ be a proper, non-compact $\text{CAT(-1)}$ space, and $\Gamma$ a discrete cocompact subgroup of the isometries of $X$. We compactify the diagonal action of $ \Gamma$ on $X\times X$ considering a domain of the horofunction boundary with respect to the maximum metric.
\end{abstract}
\maketitle
\section{Introduction}
\label{introduction}

Let $(X, d_X)$ be a proper, non-compact $\text{CAT}(-1)$ space. As examples of such spaces one can think of complete simply connected Riemannian manifolds of negative sectional curvature, for instance hyperbolic space, and of metric trees. For a review on $\text{CAT}(-1)$ spaces one can consult the first section of \cite{Bourdon}. Here, we consider a discrete and cocompact subgroup $\Gamma$ of the isometries of $(X, d_X)$, and the diagonal action of $\Gamma$ on the product space $X\times X$,
\begin{align*}
\Gamma\times X\times X &\rightarrow X\times X\\
(\gamma, x, y) &\mapsto (\gamma x, \gamma y).
\end{align*}
This action is not cocompact  but one can try to attach to $X\times X$ a set $\Omega$ of ideal boundary points such that the action on $X\times X \cup \Omega$ is cocompact. 

The space $(X\times X, d_{X\times X})$, where $d_{X\times X}$ denotes the standard product metric, is a proper metric space and hence, it can be compactified by means of horofunctions, see for instance section 3 in \cite{BGS}. Since it is a $\text{CAT}(0)$ space, the horofunctions are in fact Busemann functions and the horofunction boundary coincides with the boundary by rays \cite[Prop. 2.5]{Ballmann}. The action of $\Gamma$ extends continuously to an action by homeomorphisms on the boundary of the compactification, but it is not clear if there is a subset of the boundary that can be a good candidate for $\Omega$. For this reason we introduce the maximum metric in $X\times X$, defined by
$$d_{\max}((x,y),(x',y'))= \max\{d_X(x,x'), d_X(y,y')\}$$
for any $(x,y)$, $(x',y')$ in $X\times X$. The space $(X\times X, d_{\max})$ is also a proper metric space and therefore, it can also be compactified by horofunctions. However, it is not a $\text{CAT}(0)$ space, since the geodesic segment joining two points is not unique. The group $\Gamma$ acts on $(X\times X, d_{\max})$ by isometries and the action can  be extended to an action by homeomorphisms on the ideal boundary, which we denote by $\partial_{\infty}^{\max}(X\times X)$. The compactification with respect to the metric $d_{\max}$ turns out to be more adapted to our problem. The main result of this work is the fact that we can find a subset of $\partial_{\infty}^{\max} (X\times X)$ where the action of $\Gamma$ is properly discotinuous, and which compactifies the action of $\Gamma$ on $X\times X$:

\begin{theorem} Let $X$ be a proper, non-compact $\text{CAT}(-1)$ space and $\Gamma$ a group of isometries of $X$ acting in a properly discontinuous and cocompact way on $X$.  There exist an open set $\Omega\subset \partial_{\infty}^{\max} (X\times X)$ such that the diagonal action of $\Gamma$ on $X\times X\cup \Omega$ is properly discontinuous and cocompact.
\end{theorem}

The ideal boundary $\partial_{\infty}^{\max}(X\times X)$ can be interpreted in terms of the ideal boundary of $X$, which we denote by $\partial_{\infty}X$. We will see that $\partial_{\infty}^{\max}(X\times X)$ splits in a singular part, which is naturally homeomorphic to $\partial_{\infty}X\sqcup\partial_{\infty}X$, and a regular part, which is homeomorphic, also in a natural way, to $\partial_{\infty} X\times \partial_{\infty} X\times \mathbb{R}$. If we denote by $D$ the diagonal of $\partial_{\infty} X\times \partial_{\infty}X$, the set $\Omega$ is just the subset of the regular part of the boundary which corresponds, under the homeomorphism, to the set  $((\partial_{\infty} X \times \partial_{\infty}X)\setminus D) \times \mathbb{R}$. 

The set $\Omega$ is naturally homeomorphic to the set $G$ of parametrized geodesics in $X$ equipped with the topology of uniform convergence on compact sets. The identification gives more geometrical insight to the solution of the problem. We consider the diagonal $\Delta$ of $X\times X$ and the nearest point retraction that sends each point of $X\times X$ to its nearest point in $\Delta$. This map can be extended in a continuous way to $G$, by sending each geodesic $g$ to the point $(g(0), g(0))$ in $\Delta$. We use the continuous extension of the nearest point retraction to show that the action of $\Gamma$ on $\Omega$ is properly discontinuous and cocompact. 

\section{Definitions and notations}

In this section we review very briefly the main concepts that appear through the paper.

For us, a parametrized geodesic (or simply a geodesic) is an isometric embedding $g:\mathbb{R} \rightarrow X$. We call the image of a geodesic a geodesic line. A ray is an isometric embedding $r: [0, \infty)\rightarrow X$. And similarly, a geodesic segment joining two points $x$ and $y$ is an isometric embedding $xy:[a,b]\rightarrow X$ such that $xy(a)=x$ and $xy(b)=y$. We will make no distinction between a ray or a geodesic segment and their corresponding images. 

A geodesic space is a metric space such that any two points can be joined by a geodesic segment. A metric space is proper if its closed balls are compact. A geodesic metric space is proper if and only if it is complete and locally compact \cite[Thm. 1.10]{Gromov}.

Let $X$ be a metric space and $\Delta$ a geodesic triangle in $X$. A comparison triangle $\bar{\Delta}_{\mathbb{H}^2}$ in the  hyperbolic plane $\mathbb{H}^2$ (or $\bar{\Delta}_{\mathbb{E}^2}$ in the euclidean plane $\mathbb{E}^2$)  is a geodesic triangle in  $\mathbb{H}^2$  (or in $\mathbb{E}^2$) with sides of the same length than those of $\Delta$. The space $X$ is $\text{CAT}(-1)$ (respectively, $\text{CAT}(0)$) if for any triangle $\Delta$, for any $x, y$ in $\Delta$ and their comparison points $\bar{x}, \bar{y}$ in $\bar{\Delta}_{\mathbb{H}^2}$ (respectively, $\bar{\Delta}_{\mathbb{E}^2}$) satisfy:
$$d(x,y) \leq d(\bar{x}, \bar{y}).$$
A $\text{CAT}(-1)$ space is in particular $\text{CAT}(0)$, \cite[Part II, Thm. 1.12]{BH}.

The action of $\Gamma$ on $X$  is properly discontinuous if for every compact $K\subset X$ the set $K\cap \gamma K$ is non empty for finitely many $\gamma\in \Gamma$. The action is cocompact if there exists a compact $K\subset X$ such that $X=\Gamma K$.

Given a metric space $(X,d)$, the Gromov product of two points $x$ and $y$ in $X$ with respect to a third point $z$ in $X$  is defined as:
$$(x|y)_z= \frac{1}{2}(d(z,y) + d(z,x) - d(x,y)).$$
In a $\text{CAT}(-1)$ space the product can be extended to points at infinity. Observe that a geodesic $g$ determines two rays, one with the same orientation than the geodesic itself and the other with the reversed orientation. The equivalence classes of these rays are the ideal endpoints of the geodesic, we denote them by $g(+\infty)$ and $g(-\infty)$ respectively. The Gromov product of the points $g(+\infty)$ and $g(-\infty)$ with respect to a base point $o$, is given by:
$$(g(+\infty)|g(-\infty))_o= \frac{1}{2}(\beta^ {g(0)}_{g(+\infty)} (o) + \beta^{g(0)}_{g(-\infty)}(o)),$$
where the functions $\beta^ {g(0)}_{g(+\infty)}$, $\beta^ {g(0)}_{g(-\infty)}$ will be defined in section $3$. The definition of the Gromov product is independent of the parametrization of the geodesic line $L$ with ideal endpoints $g(+\infty)$ and $g(-\infty)$. Observe that if $o\in L$ then it is $0$. (See for instance \cite[Ch. 2]{Ghys} and \cite[Sect. 2]{Bourdon}, for more about Gromov products). 

\section{Horofunction boundary of $(X \times X, d_{\max})$}
\label{compactification}
As we have stated in the introduction, let $(X, d_X)$ be a proper, non-compact $\text{CAT}(-1)$ space and consider the product space $X\times X$ together with the maximum metric $d_{\max}$:
$$d_{\max}((x,y),(x',y'))= \max\{d_X(x,x'), d_X(y,y')\}$$
for any $(x,y)$, $(x',y')$ in $X\times X$.

The compactification via horofunctions of a proper metric space is explained in detail in \cite[Ch.2]{Ballmann} and \cite[Part II, Sect.8]{BH}.  The idea, which is in fact valid for any complete locally compact metric space, is due to Gromov \cite[Sect.3]{BGS}. It consists on embedding the space, in our case $(X\times X, d_{\max})$, into the space $C_*$ of its continuous functions (with the topology of uniform convergence on compact sets) modulo additive constant, via the map
 \begin{align*}
\iota: X\times X &\rightarrow C_*\\
 (x,y) &\mapsto [d_{\max}((x,y), \cdot)]
\end{align*}
 that assigns to each point in the space the class in $C*$ of the distance function with respect to this point. \\
The closure of the space, denoted by $\overline{X\times X}^{\max}$, is the closure of $\iota(X\times X)$ in $C_{\ast}$, and the ideal boundary, denoted by $\partial_{\infty}^{\max}(X \times X)$, is the set $(\overline{X\times X}^{\max})\setminus \iota(X\times X)$.  Both $\overline{X\times X}^{\max}$ and $\partial_{\infty}^{\max}(X\times X)$ are compact since $X\times X$ is locally compact. 

A horofunction is a continuous function such that its class belongs to $\partial_{\infty}^{\max}( X \times X)$. One can think of horofunctions as limits of normalized distance functions. The level sets of a horofunction are known as horospheres and the sublevel sets as horoballs. Observe that two horofunctions in the same equivalence class differ by a constant and share the same set of horospheres and horoballs.

We can compactify in the same way the original space $(X, d_X)$, since it is also a proper metric space. In this case, since $(X, d_X)$ is $\text{CAT}(-1)$, the compactification obtained is homeomorphic to the compactification by rays, where the points at infinity are equivalence classes of rays, and two rays $c$, $c'$ are in the same class if $d_X(c(t), c'(t))\leq C$ for all $t$. The equivalence class of a ray $c$ is denoted by $c(\infty)$, the ideal boundary of $X$ by $\partial_{\infty}X$ and an arbitrary point in $\partial_{\infty} X$ by $\xi$. The horofunctions  of a CAT(0) space are in fact Busemann functions and can be written as:
$$\beta_{\xi}^p (\cdot)= \lim_{t\rightarrow\infty} d_{X}(\cdot , c(t))-t,$$
where $c$ is a ray such that $c(0)=p$ and $c(\infty)= \xi \in \partial_{\infty}X$. Observe that $\beta^{p}_{\xi}$ is the representative of $\xi\in \partial_{\infty}X$ that satisfies $\beta^{p}_{\xi}(p)=0$.
Two properties we will use about Busemann functions are:
\begin{itemize}
\item If a sequence $\{x_n\}_n$ converges to a point $\xi$ in $\partial_{\infty}X$, then \cite[Prop. 2.5]{Ballmann}
$$\beta^o_{\xi}(y)=\lim_{n\rightarrow\infty} d(x_n, y) - d(x_n, o).$$
\item $ \beta^o_{\xi}(y) = -\beta^y_{\xi}(o)$ \cite[Sect.2.1]{Bourdon}.
\end{itemize}
\begin{example}
\label{boundary of Hn}
 Let $X$ be a complete simply connected n-dimensional Riemannian manifold $X$ of sectional curvature $\leq -1$, then $\partial_{\infty} X\cong S^{n-1}$. In particular, $\partial_{\infty}\mathbb{H}^n \cong S^{n-1}\cong (T_{x_0}\mathbb{H}^n)^{1}$ for any $x_0\in X$.
 \end{example}

Now we list all the possible boundary points obtained when compactifying $(X \times X, d_{\max})$. For the calculations, we have chosen a base point $O=(o, o)$ with $o \in X$ and then, as a representative of a class of distance functions the function 
$$d_{\max}^{O}((x,y), \cdot)= d_{\max}((x,y), \cdot) - d_{\max}((x,y), (o,o)).$$
Then $\{[d_{\max}(P_n, \cdot)]\}_n\rightarrow \xi$ if and only if $\{d_{\max}^{O}(P_n, \cdot)\}_n \rightarrow h^{O}_{\xi}$, where $h^{O}_{\xi}$ is the horofunction in $\xi$ that satisfies $h^{O}_{\xi}(O)=0$.
\begin{proposition}
\label{boundary points}
Every divergent sequence $\{(x_n, y_n)\}_ n \subset (X \times X, d_{\max})$ has a subsequence that satisfies, except for permutations of $x_n$ and $y_n$, one of the possibilities below :
\begin{enumerate}[label=(\Roman*)]
\item $d_{X}(x_n, o) \text{ is bounded for all n and } \{y_n\}_n\rightarrow \xi' \in \partial_{\infty}X$.
\item $\{x_n\}_n\rightarrow \xi\in \partial_{\infty}X \text{, } \{y_n\}_n \rightarrow \xi'\in \partial_{\infty}X \text{ and } d_{X}(x_n, o)- d_{X}(y_n, o)\rightarrow C $. 
\item $\{x_n\}_n\rightarrow \xi\in \partial_{\infty}X \text{, } \{y_n\}_n\rightarrow \xi'\in \partial_{\infty}X \text{ and } d_{X}(x_n, o)- d_{X}(y_n, o) \rightarrow +\infty$.  
\end{enumerate}
For each of the possibilities, the limit of the subsequence is 
\begin{enumerate}[label=(\Roman*)]
\item $\lim_{n\rightarrow \infty} d_{\max}^{O}((x_n,y_n),(z,z'))  = \beta^{o}_{\xi'}(z')$.
\item $\lim_{n\rightarrow \infty} d_{\max}^{O}((x_n,y_n),(z,z'))= \max\{\beta^{o}_{\xi}(z), \beta^{o}_{\xi'}(z')+C\}$.
\item $\lim_{n\rightarrow \infty} d_{\max}^{O}((x_n,y_n),(z,z'))= \beta^{o}_{\xi}(z)$.
\end{enumerate}
For every $(z,z')$ in $X\times X$.
\end{proposition}
\begin{proof}
We do the proof for case (II), the other two cases can be obtained in a similar fashion. Denote $C_n= d(x_n,o) - d(y_n,o)$ and assume to simplify $C_n\geq 0$. Then
\begin{multline*}
d_{\max}((x_n, y_n), (x,y))- d_{\max}((x_n,y_n), (o,o)) =\\  \max\{d(x_n,x) - d(x_n,o), d(y_n,y)-d(y_n,o)-C_n\}.
\end{multline*}
Fix an $\epsilon>0$ and an $r>0$, we want to see that there is an $N$ such that for all $n>N$ and for all $(x,y)\in B^{\max}(O,r)$,
\begin{equation*}
|\max\{d(x_n,x)-d(x_n,x), d(y_n, y) - d(y_n,o)-C_n\}-\max\{\beta_{\xi}^o(x), \beta^o_{\xi'}(y)-C\}|<\epsilon,
\end{equation*}
where $B^{\max}(O,r)$ is the ball of center $O=(o,o)$ and radius $r$ in $(X\times X, d_{\max})$.
There are four cases to check. We do the case for which the first maximum is $d(x_n,x) - d(x_n,o)$ and the second maximum is $\beta^o_{\xi'}(y)-C$, the other are similar. Because of the definition of Busemann function, there is an $N$ such that for all $(x,y)\in B^{\max}(O, r)$, and all $n>N$  
$$d(x_n,x)-d(x_n,o) -\beta^o_{\xi}(x)<\epsilon,$$
$$d(y_n,y) - d(y_n,o) - C_n - (\beta_{\xi'}^o(y') -C)>-\epsilon.$$
We have on the one hand:
\begin{eqnarray*}
d(x_n,x)- d(x_n, o)-(\beta^o_{\xi'}(y)-C)&=&
d(x_n,x)-d(x_n,o) -\beta^o_{\xi}(x) \\&+& \beta^o_{\xi}(x) -(\beta_{\xi'}^o(y)-C)<\epsilon
\end{eqnarray*}
and 
\begin{multline*}
d(x_n,x)-d(x_n,o) -(\beta_{\xi'}^o(y)-C)=
d(x_n,x)-d(x_n,o) -(d(y_n,y)-d(y_n,o) -C_n)+\\
d(y_n,y) - d(y_n,o) - C_n - (\beta_{\xi'}^o(y') -C)> -\epsilon  
\end{multline*}
\end{proof}

We call the set of boundary points with a representative of the form $\beta^{o}_{\xi}(z)$ or $\beta^{o}_{\xi'}(z')$ the singular part of the boundary and we denote them by $\partial_{\infty}^{\max}(X \times X)_{\text{sing}}$. The rest of the points, those with a representative of the form $\max\{\beta^{o}_{\xi}(z), \beta^{o}_{\xi'}(z')+C\}$, define the regular part of the boundary and we denote them by $\partial_{\infty}^{\max}(X \times X)_{\text{reg}}$. 

Observe that for any constant $C'$, the function $\max\{\beta^{o}_{\xi}(z), \beta^{o}_{\xi'}(z')+C\}+C'=\max\{\beta^{o}_{\xi}(z)+ C', \beta^{o}_{\xi'}(z')+C+C'\}$ is in the same class as the function $\max\{\beta^{o}_{\xi}(z), \beta^{o}_{\xi'}(z')+C\}$. Since two Busemann functions of $X$ associated to the same point $\xi\in\partial_{\infty}X$ differ by a constant, for each $C'$ we can find points $p$ and $p'$ in $X$ such that $\beta^p_{\xi}(z)=\beta^{o}_{\xi}(z)+ C'$ and $\beta^{p'}_ {\xi'}(z')=\beta^{o}_{\xi'}(z')+C+C'$. So the regular points are in fact the classes modulo constant of the functions $\max\{\beta^p_{\xi}(z), \beta^{p'}_{\xi'}(z')\}$ for all $p, p'\in X$ and $\xi, \xi'\in \partial_{\infty}X$.

\begin{proposition}
\label{homeo singular}
 There is a natural homeomorphism 
$$\varphi_{sing}: \partial_{\infty}^{\max}(X \times X)_{\text{sing}} \longrightarrow \partial_{\infty}X \sqcup \partial_{\infty} X $$
that consists in associating to a Busemann function that takes values only in the first (second) factor of $X \times X$ the same Busemann function viewed as a point of the first (second)  factor  in $\partial_{\infty}X \sqcup \partial_{\infty} X$.
\end{proposition}
\begin{proof}
It follows from the fact that the set of Busemann functions in one factor is naturally identified to the boundary of $X$.
\end{proof}
The regular part of the boundary can also be identified with a more easy to handle object. 
\begin{proposition}
\label{homeo regular}
For each choice of base point $(o,o')\in X\times X$ there is a natural homeomorphism 
\begin{equation}
\begin{array}{r@{}l}
\varphi_{reg}: \partial_{\infty}^{\max}(X \times X)_{\text{reg}} &{}\longrightarrow  \partial_{\infty} X \times \partial_{\infty} X\times \mathbb{R} \\ 
\left[ \max\{\beta_{\xi}^{p}(z), \beta^{p'}_{\xi'}(z')\}\right]  &{} \mapsto (\xi, \xi', \beta^{p'}_{\xi'}(o)-\beta^{p}_{\xi}(o')).
\end{array}
\end{equation}
\end{proposition}
\begin{remark}
\label{homeo regular with a base point}
Under our choice of base point the homeomorphism $(1)$ takes the form:
 \begin{equation}
\begin{array}{r@{}l}
\varphi_{reg}: \partial_{\infty}^{\max}(X \times X)_{\text{reg}} &{}\longrightarrow  \partial_{\infty} X \times \partial_{\infty} X\times \mathbb{R} \\ 
\max\{\beta^{o}_{\xi}(z), \beta^{o}_{\xi'}(z')+C\} &{}\mapsto (\xi, \xi', C).
\end{array}
\end{equation}
\end{remark}
\begin{proof}(of Prop. 3)
The map (1) is well defined since two horofunctions in the same class differ by a constant and $\max\{\beta^p_{\xi}, \beta^{p'}_{\xi'}\}\neq \max\{\beta^q_{\eta}, \beta^{q'}_{\eta'}\}$ for $(\xi,\xi')\neq(\eta, \eta')$.  Two see this, normalize the Busemann functions with respect to the same point; i.e. $\max\{\beta^p_{\xi}, \beta^{p'}_{\xi'}\}=\max\{\beta^p_{\xi}, \beta^p_{\xi'}+A\}$ and $\max\{\beta^q_{\eta}, \beta^{q'}_{\eta'}\}=\max\{\beta^p_{\eta}+B, \beta^p_{\eta'}+C\}$ for some constants $A$, $B$, $C$. Choose a sequence $z_n\rightarrow \xi$ along the geodesic ray joining $p$ and $\xi$. Then $\beta^p_{\xi}(z_n)\rightarrow -\infty$ and $\beta^p_{\nu}(z_n)\rightarrow +\infty$ for all $\nu\in\partial_{\infty} X$ such that $\nu\neq \xi$. Using this property one can see that $\xi=\eta$ and $\xi'=\eta'$ if $\max\{\beta^p_{\xi}, \beta^{p'}_{\xi'}\} = \max\{\beta^q_{\eta}, \beta^{q'}_{\eta'}\}$.

Now for each class we choose the representative of the form $\max\{\beta^{o}_{\xi}(z), \beta^{o}_{\xi'}(z')+C\}$ and prove that the map (2) is a homeomorphism. Injectivity is clear. For the exhaustivity, given $(\xi, \xi', C)$ one can see that the sequence $(g(n), g(-n))$ such that $g$ is a parameterization of the geodesic line joining $\xi'$ and $\xi$ with $\beta^o_{\xi}(g(0))- \beta^o_{\xi'}(g(0))=C$ has as a limit $\max\{\beta^o_{\xi}, \beta^o_{\xi'}+C\}$. 

For the continuity, take a sequence $\max\{\beta^o_{\xi_n}, \beta^o_{\xi'_n}+C_n\}$ that converges to a point $\max\{\beta^o_{\xi}, \beta^o_{\xi'}+C\}$. The $C_n$ must be bounded, otherwise the sequence converges to a Busemann function in one factor. This and the compactness of the set of Busemann fuctions of $X$, implies that $(\xi_n, \xi'_n, C_n)$ converges to some point $(\eta, \eta', C')$. Since the maximum function is continuous, $\max\{\beta^o_{\xi_n}, \beta^o_{\xi'_n}+C_n\}$ should also converge to $\max\{\beta^o_{\eta}, \beta^o_{\eta'}+C\}$. Therefore, $\max\{\beta^o_{\xi}, \beta^o_{\xi'}+C\}=\max\{\beta^o_{\eta}, \beta^o_{\eta'}+C\}$ and $(\xi, \xi', C)= (\eta, \eta', C')$. The continuity of the maximum function also assures that the inverse of (2) is continuous.

\end{proof}
\begin{example}
For $X$  a complete simply connected n-dimensional Riemannian manifold of sectional curvature $\leq -1$, $\partial_{\infty}^{\max}(X \times X)_{reg} \cong S^{n-1}\times S^{n-1} \times \mathbb{R}$ and $\partial_{\infty}^{\max}(X \times X)_{sing} \cong S^{n-1}\sqcup S^{n-1}$. It can be shown that the boundary of $X\times X$ is homeomorphic to a $(2n-1)$-sphere:
$$\partial_{\infty}^{\max}(X \times X)\cong \join(S^{n-1}, S^{n-1}) \cong  S^{2n-1}.$$ 
\end{example}
\section{An ideal domain for the action of $\Gamma$}
\label{ideal domain}

From now on, $\Gamma$ will be a discrete and cocompact subgroup of the isometries of $X$. Recall that we were interested in the diagonal action of $\Gamma$ on $X\times X$. In this section we look for an open subset of $\partial_{\infty}^{\max}(X\times X)$ where the action of $\Gamma$ is good enough. 

Let $D$ denote the diagonal in $\partial_{\infty}X \times \partial_{\infty}X$:
 $$D = \{(\xi, \xi)\text{; } \xi \in \partial_{\infty}X \}\cong \partial_{\infty} X,$$
 and $\Lambda_{(x,y)}$, the limit set of the orbit of a point $(x,y)$ in $X\times X$:
$$\Lambda_{(x,y)}= \{\overline{\Gamma(x,y)} \cap \partial^{\max}_{\infty} (X\times X)\}.$$
Notice that $\Lambda_{(x,y)}$ depends on $(x,y)\subset X\times X$ since this space is not $\text{CAT}(-1)$. We define the limit set $\Lambda_{\Gamma}$ of $\Gamma$ as
$$\Lambda_{\Gamma}= \bigcup_{(x,y)\in X\times X} \Lambda_{(x,y)}.$$
\begin{proposition} 
\label{limit set}
The diagonal action of $\Gamma$ on $X\times X$ satisfies 
\begin{enumerate}
\item $\Lambda_{\Gamma} \subset \partial_{\infty}^{\max}(X \times X)_{\text{reg}}$
\item $\varphi_{reg}(\Lambda_{\Gamma}) = D \times \mathbb{R}$
\end{enumerate}
\end{proposition}
\begin{proof}
First observe that the limit of any sequence $(\gamma_n x, \gamma_n y)$ is in $D\times \mathbb{R}$. Indeed, by the triangle inequality, $|d(\gamma_nx, o)- d(\gamma_n y,o)|\leq d(x,y)$, so the limit is a regular point, and since $d(\gamma_nx, \gamma_n y)= d(x,y)$, if $\gamma_n x\rightarrow \xi$ then $\gamma_n y\rightarrow \xi$ because $X$ is $\text{CAT}(-1)$.

Let us see next that any point $(\xi, \xi, C)$ is in the limit set. Let $\gamma_n$ be a sequence in $\Gamma$ such that $\gamma_n \rightarrow \xi$. Observe that such a sequence exists since $\Gamma$ is cocompact and hence its limit set in $\overline{X}$ is the whole $\partial_{\infty}X$. Let $\xi'\in \partial_{\infty} X$ be the limit of the sequence $\gamma^{-1}_n$ and take any point $(x,y)$ satisfying $\beta^{o}_{\xi'}(x) - \beta^{o}_{\xi'}(y)=C$. For instance, one can take a point $(g(t), g(t'))$, where $g$ is the ray joining $o$ and $\xi'$ and $t-t'=C$. Then:
\begin{eqnarray*}
\lim_{n\rightarrow\infty} d(\gamma_n x, o) - d(\gamma_n y, o) &=& \lim_{n\rightarrow\infty} d(x, \gamma^{-1}_n o) - d(\gamma_n^{-1}o, o) - (d(y, \gamma^{1}_n o)- d(\gamma_n^{-1}o, o))\\ &=&\beta^{o}_{\xi'}(x) - \beta^{o}_{\xi'}(y)=C.
\end{eqnarray*}
Hence, the limit of the sequence $(\gamma_n x, \gamma_ny)$ is the point $(\xi, \xi, C)$.
\end{proof}

We choose, as a candidate for the domain at infinity, the set $\Omega\subset\partial_{\infty}^{\max} (X \times X)_{\text{reg}}$ such that:
$$\Omega \cong (\partial_{\infty} X \times \partial_{\infty}X \setminus D)\times \mathbb{R}$$
under the homeomorphism (1). Observe that we have excluded the whole region $D\times\mathbb{R}$.

Now consider the set:
 $$G= \{\text{parametrized geodesics in }X \}$$ 
This set is the same as the set of oriented geodesic lines with a distinguished base point and as we show next, it is in correspondence with the points of $\Omega$. Observe that there is a natural action of $\Gamma$ on $G$: an element $\gamma\in\Gamma$ sends a geodesic $g$ to a geodesic $\gamma g$.
\begin{lemma}
\label{ideal bijection}
The map
\begin{align*}
f: G &\longrightarrow \Omega\\
 g &\mapsto \lim_{n\rightarrow \infty} (g(n), g(-n))
\end{align*}
is a bijection.
\end{lemma}
\begin{proof}
First of all we check that given a geodesic $g$ in $G$, the limit of the sequence $\{(g(n), g(-n))\}_n$ belongs to $\Omega$. In order for this limit to be a regular point of the boundary, the sequence $\{(g(n), g(-n))\}_n$ needs to belong to case $(II)$ of Proposition \ref{boundary points}, so the limit of the difference $d_X(g(n), o) - d_X(g(-n), o)$ has to be a real constant. This follows from the next calculation:
\begin{eqnarray*}
 \lim_{n\rightarrow \infty} d_X(g(n), o) - d_X(g(-n), o) &=&
\lim_{n\rightarrow \infty} d_X(g(n), o) - d_X(g(n),g(0)) \\&& + d_X(g(-n), g(0))  - d_X(g(-n), o) 
\\&=& \beta^ {g(0)}_{g(+\infty)} (o) + \beta^{g(0)}_{g(-\infty)}(o)\\&=&
\beta^{o}_{g(-\infty)}(g(0)) - \beta^{o}_{g(+\infty)}(g(0)) \in\mathbb{R}.
\end{eqnarray*}
Here we have used the fact that $d_X(g(n), g(0))= d_X(g(-n), g(0))$ and the definition of Busemann function.
Henceforth, the limit of the sequence $\{(g(n), g(-n))\}_n$ is the point
$$(g(+\infty), g(-\infty), \beta^{o}_{g(-\infty)}(g(0)) - \beta^{o}_{g(+\infty)}(g(0)))$$
in $\partial_{\infty} X \times \partial_{\infty}X \times \mathbb{R}$ under the homeomorphism $(2)$. Since $g(+\infty)\neq g(-\infty)$, this limit belongs to $\Omega$.

Let us check the injectivity of the map $f$. Suppose we have two different geodesics $g$ and $g'$ such that $f(g)=f(g')$. Then, $g(+\infty)=g'(+\infty)$ and $g(-\infty)= g'(-\infty)$ and since given two ideal points in a $\text{CAT}(-1)$ space there is a unique geodesic line having them as ideal endpoints \cite[Thm. 9.33]{BH}, both geodesics must be different parametrizations of the same geodesic line $L$. 
Now, observe that the difference $\beta^{o}_{g(-\infty)}(g(0)) - \beta^{o}_{g(+\infty)}(g(0))$ can be rewritten using the Gromov product $(g(+\infty)|g(-\infty))_{o}$:
\begin{eqnarray*} 
\beta^{o}_{g(-\infty)}(g(0)) - \beta^{o}_{g(+\infty)}(g(0))&=&\beta^ {g(0)}_{g(+\infty)} (o) - \beta^{g(0)}_{g(-\infty)}(o))=\\ &=& 2(\beta^{g(0)}_{g(+\infty)}(o) - (g(+\infty)|g(-\infty))_{o})=\\&=&
-2(\beta^{o}_{g(+\infty)}(g(0)) + (g(+\infty)|g(-\infty))_{o}).
\end{eqnarray*}
Therefore the two parametrizations satisfy:
$$-2(\beta^{o}_{g(+\infty)}(g(0)) + (g(+\infty)|g(-\infty))_{o})= -2(\beta^{o}_{g'(+\infty)}(g'(0)) + (g'(+\infty)|g'(-\infty))_{o}) $$
and hence $\beta^{o}_{\xi}(g(0)) = \beta^{o}_{\xi}(g'(0))$. Since both $g(0)$ and $g'(0)$ belong to $L$, we must have  $g(0)=g'(0)$, so both parametrizations of $L$ are the same.

To finish, we check the exhaustivity of $f$. Let $(\xi_+, \xi_-, r)$ be a point in $\Omega$ (seen through the homeomorphism $(1)$). We are looking for a geodesic $g$ such that $\lim_{n\rightarrow \infty} \{(g(n), g(-n))\}_n$ is $(\xi_+, \xi_-, r)$. We have already calculated the limit of such a sequence at the beginning of the proof and we know it is the point $(g(+\infty), g(-\infty), \beta^{o}_{g(-\infty)}(g(0)) - \beta^{o}_{g(+\infty)}(g(0)))$. So we look for a geodesic such that $g(+\infty)= \xi_+$, $g(-\infty)=\xi_-$ and $\beta^{o}_{g(-\infty)}(g(0)) - \beta^{o}_{g(+\infty)}(g(0))=r$. Let $L$ be the geodesic line with ideal endpoints $\xi_+$ and $\xi_-$. Consider a parametrization $g(t)$ of $L$. Since the functions $\beta^{o}_{\xi_+}(g(t))$ and $\beta^{o}_{\xi_-}(g(t))$ are lineal with slope $\pm 1$ respectively, there is a unique point $p$ in $L$ that satisfies
$$\beta^{o}_{\xi_-}(p)- \beta^{o}_{\xi_+}(p)=r.$$
The parametrization $g'(t)$ of $L$ such that $g'(0)=p$ is the one we are looking for, since it satisfies  
$g'(+\infty)= \xi_+$, $g'(-\infty)=\xi_-$ and $\beta^{o}_{g'(-\infty)}(g'(0)) - \beta^{o}_{g'(+\infty)}(g'(0)))=r$.

\end{proof}
\begin{remark} Observe that what we have checked in the proof of Lemma \ref{ideal bijection}  is in fact the bijectivity of the map $\varphi_{reg}\circ f$.
\end{remark}

We consider in $G$ the topology of uniform convergence on compact sets. 
\begin{theorem}
\label{ideal homeomorphism}

The map
\begin{align*} 
f: G &\longrightarrow \Omega\\
 g &\mapsto \lim_{n\rightarrow \infty} (g(n), g(-n))
\end{align*}
is an equivariant homeomorphism.
\end{theorem}
\begin{proof}
To see that $f$ is a homeomorphism consider the map 
\begin{align*}
f': G &\rightarrow ((\partial_{\infty} X \times \partial_{\infty} X )\setminus D) \times \mathbb{R}\\
g &\mapsto (g(+\infty), g(-\infty), \beta_{g(+\infty)}^{g(0)}(o)- \beta_{g(-\infty)}^{g(0)}(o))
\end{align*}
which is just $\varphi_{reg} \circ f$ and which, as we have seen along the proof of Lemma \ref{ideal bijection}, is a bijection.
Let $\phi$ be the Hopf parametrization 
\begin{align*}
\phi: G&\rightarrow ((\partial_{\infty} X \times \partial_{\infty} X) \setminus D) \times \mathbb{R})\\
g&\mapsto (g(+\infty), g(-\infty), \beta^{g(0)}_{g(+\infty)}(o))
\end{align*}
which is a homeomorphism (see \cite[Sect. 2.9]{Bourdon}), and let $h$ be the map 
\begin{align*}
h: ((\partial_{\infty} X \times \partial_{\infty}  X) \setminus D) \times \mathbb{R})&\rightarrow ((\partial_{\infty} X \times \partial_{\infty} X) \setminus D) \times \mathbb{R})\\
(\xi_+, \xi_-, r) &\mapsto (\xi_+, \xi_-, 2(r-(\xi_+|\xi_-)_{O}))
\end{align*}
which is also a homeomorphism.

The following diagram commutes
\begin{center}
\xymatrix{
G \ar[r]^{f'} \ar[d]^{\phi} & ((\partial_{\infty} X\times \partial_{\infty} X) \setminus D)\times\mathbb{R})\\
((\partial_{\infty} X\times \partial_{\infty} X) \setminus D)\times\mathbb{R}) \ar[ru]^{h}}
\end{center}
Indeed, for any $g\in G$
\begin{eqnarray*}
h\circ\phi(g)&=& h((g(+\infty), g(-\infty), \beta^{g(0)}_{g(+\infty)}(o)))=\\
&=& (g(+\infty), g(-\infty), 2(\beta^{g(0)}_{g(+\infty)}(o)-(g(+\infty)|g(-\infty))_o)=\\
&=& (g(+\infty), g(-\infty), \beta ^{g(0)}_{g(+\infty)}(o)-\beta ^{g(0)}_{g(-\infty)}(o))=f'(g).
\end{eqnarray*}

Now, since both $\phi$ and $g$ are homeomorphisms, $f'$ is a homeomorphism. And since $f'=\varphi_{reg}\circ f$ and $\phi_{reg}$ is a homeomorphism, our map $f$ is a homeomorphism.

To finish, observe that the map $f$ is equivariant since given $\gamma\in \Gamma$
$$ f(\gamma g)= \lim_{n\rightarrow \infty} (\gamma g(n), \gamma g(-n))= \gamma\lim_{n\rightarrow \infty} (g(n), g(-n)) = \gamma f(g) $$
for any $g$ in $ G$.
\end{proof}
The next proposition is a consequence of Theorem \ref{ideal homeomorphism} and the properties of $G$. It will also be a consequence of Theorem \ref{quotient compact}.
\begin{proposition}
\label{action ideal domain}
The action of $\Gamma$ on $\Omega$ is cocompact and properly discontinuous.
\end{proposition}

\section{Compactness of $(X\times X\cup \Omega)/\Gamma$}
\label{compactness}
We define a topology in $X\times X\cup G$ in the following way. We keep the same topology in $X\times X$ and $G$ and we say a sequence $\{(x_n, y_n)\}_n$ in $X\times X$ converges to a point $g$ in $G$ if and only if:
\begin{enumerate}[label=(\Roman*)]
\item $\{x_n\}_n \rightarrow g(+\infty).$ 
\item $\{y_n\}_n \rightarrow g(-\infty).$
\item $d_X(x_n, o) - d_X(y_n,o) \rightarrow \beta_{g(-\infty)}^o(g(0)) - \beta_{g(+\infty)}^o(g(0)).$
\end{enumerate} 
With this topology $X\times X\cup G$ and $X\times X \cup \Omega$ are homeomorphic.

Now, let the diagonal in $X\times X$ be the set
$$\Delta= \{(x,x)\text{; } x\in X\}$$
and let $\rho: X\times X \rightarrow \Delta$ be the map that sends each point in $X\times X$ to its nearest point in $\Delta$ with respect to the metric $d_{\max}$. 
\begin{lemma}
\label{fibres}
The fibre $\rho^{-1}(a,a)$ is the set of point $(x,y)\in X\times X$ such that $a$ is the midpoint of the segment $xy$.
\end{lemma}
\begin{definition}
We extend the projection $\rho$ to a map:
$$\tilde{\rho}: X \times X \cup G \rightarrow \Delta$$
as follows:
\begin{itemize}
\item If $p=(x,y)\in X\times X$, then $\tilde{\rho}(p)= \rho(p)$ .
\item If $g\in G$, then $\tilde{\rho}(g)=(g(0), g(0))$.
\end{itemize}
\end{definition}

The next step we want to undertake is to prove that this extension is continuous. Before, we need a couple auxiliary lemmas. 
\begin{lemma} 
\label{lema2}
In a proper $\text{CAT}(-1)$ space
$$\lim_{i,j}(x_i|y_j)_o= (\xi|\xi')_o,$$
for any sequences $x_i\rightarrow \xi$, $y_j\rightarrow \xi'$ \cite[Prop 3.4.2]{Schroeder}. Moreover, $(\xi|\xi')_o= +\infty$ if and only if $\xi=\xi'$ \cite[Ch.III.H. Rmk. 3.17]{BH}. 
\end{lemma}
\begin{lemma}
\label{geodesics}
Given a sequence of geodesic segments $g_n:[a_n, b_n]\rightarrow X$ such that $g_n(a_n) \rightarrow \xi$, $g_n(b_n) \rightarrow \xi'$ and $g_n(0)\rightarrow m\in X$, there exist a convergent subsequence to a geodesic $g$ satisfying $g(+\infty)=\xi$, $g(-\infty)=\xi'$ and $g(0)=m$.
\end{lemma}
\begin{proof} The fact that the sequence $\{g_n\}_n$ converges to a geodesic $g$ such that $g(0)=m$ is a consequence of Arzela-Ascoli \cite[Thm. 1.4.9]{Papadopoulos} for proper metric spaces. Now, for any $t\in\mathbb{R}$ observe that
$(g_n(a_n)|g_n(b_n))_{g_n(t)}=0$, since for all $n$, $g_n(t)$ is a point of the segment $g_n$, and $\lim_n (g_n(a_n)|g_n(b_n))_{g_n(t)} \rightarrow (\xi|\xi')_{g(t)}$ by the continuity of the Gromov product. Therefore for all $t$, $(\xi|\xi')_g(t)=0$ and $g(t)$ belongs to the line joining $\xi$ and $\xi'$. Since this line is unique, $g(+\infty)=\xi$ and $g(-\infty)=\xi'$.

\end{proof}

\begin{proposition}
\label{continuous retraction}
The map
$$\tilde{\rho}: X\times X \cup G \rightarrow \Delta$$
is continuous.
\end{proposition}
\begin{proof}
The restrictions of $\tilde\rho$  to $X\times X$ and to $G$ are continuous.  Let $\{(x_n, y_n)\}_n$ be a sequence in $X\times X$ that converges to a geodesic $g$ in $G$, so $x_n\rightarrow g(+\infty)$, $y_n\rightarrow g(-\infty)$  and $d(x_n,o)-d(y_n, o) \rightarrow C$. 
First of all, we will prove that the geodesic segments $x_ny_n$ converge to a parameterization of the geodesic line $L$ with ideal endpoints $\xi=g(+\infty)$ and $\xi'=g(-\infty)$ and then we will see that this parameterization is precisely $g$. 

For each pair $(x_n, y_n)$ let $m_n$ be the middle point of the segment $x_ny_n$. The points $m_n$ lie in a compact set. Indeed, suppose $d(o,m_n)\rightarrow +\infty$ so $m_n\rightarrow \eta\in \partial_{\infty}X$. Using the definition of Gromov product:
$$(x_n|m_n)_o = \frac{1}{2}((x_n|y_n)_o+ d(m_n,o) +\frac{1}{2}(d(x_n,o)-d(y_n,o)))$$
Since $\xi\neq \xi'$, by Lemma \ref{lema2} $\lim_n (x_n|y_n)_o$ is bounded. By hypothesis, $d(x_n,o)-d(y_n,o)$ is also bounded and $d(m_n,o)\rightarrow +\infty$. Therefore, $\lim_n (x_n|m_n)_o=+\infty$ and by Lemma \ref{lema2} again, $\eta=\xi$. Similarly, one could find that $\eta= \xi'$, so $\xi=\xi'$ and  arrive to a contradiction. Therefore $m_n\rightarrow m$ for some $m\in X$ and $m$ must be a point in $L$ by Lemma \ref{geodesics}.

Now we have, on one hand, 
$$d_X(m_n, y_n)-d_X(o, y_n) - (d_X(m_n, x_n) -d_X(o, x_n)) = d_X(o,x_n)- d_X(o, y_n) \rightarrow C.$$
On the other hand, since $d_X(\cdot, y_n)-d_X(o, y_n) \rightarrow \beta^o_{\xi'}(\cdot)$ and $d_X(\cdot, x_n)-d_X(o, x_n) \rightarrow \beta^o_{\xi}(\cdot)$ uniformly on compact sets, we have 
$$d_X(m_n, y_n)-d_X(o, y_n) - (d_X(m_n, x_n) -d_X(o, x_n)) \rightarrow \beta^o_{\xi'}(m)-\beta^o_{\xi}(m)$$ 
so
$$\beta^o_{\xi'}(m)-\beta^o_{\xi}(m)=C.$$
But the only point in $L$ that satisfies this equation is precisely $g(0)$. Hence $m=g(0)$.

Therefore, $\tilde{\rho}(x_n, y_n) = (m_n, m_n) \rightarrow (m, m)= (g(0), g(0)) = \tilde{\rho}(g)$ and the map $\tilde{\rho}$ is everywhere continuous. 
\end{proof}
\begin{corollary}
\label{boundary of a fibre}
The fibre $\tilde{\rho}^{-1}(x,x)$ restricted to $G$ is the set $G_x\subset G$ such that 
$$G_{x}= \{ \sigma \in G |\sigma(0)= x\}.$$
\end{corollary}
\begin{example}
\label{unitary tangent}
In a Riemannian manifold $X$ of dimension $n$ and sectional curvature $\leq -1$, $G_x$ is identified with the unitary tangent at $x$, $(T_{x}X)^1 \cong S^{n-1}$.
\end{example}
Observe that since $X\times X \cup G$ and $X \times X \cup \Omega$ are homeomorphic we have also a continuous projection from $X\times X\cup \Omega$ to $\Delta$ that we also call $\tilde{\rho}$. 

Now, let $K\subset X$ be a compact in $X$ such that $K/\Gamma\cong X/\Gamma$ and consider the set 
$$K_{\Delta}=\{(x,x) \in \Delta \text{ such that } x \in K \}.$$
\begin{lemma}
\label{compact domain}
The set $\tilde{\rho}^{-1}(K_{\Delta})$ is compact.
\end{lemma}
\begin{proof}

Consider a sequence $\{(x_n, y_n)\}_n$ in $\tilde{\rho}^{-1}(K_{\Delta})$. Then, since $K_{\Delta}$ is compact, $\{\tilde{\rho}((x_n, y_n))\}_n$ has a convergent subsequence $\{(m_n, m_n)\}_n$ in $K_{\Delta}$. Take a sequence of points in $\{(x_n, y_n)\}_n$ that project to this convergent subsequence.  They have a subsequence $\{(x'_n, y'_n)\}_n$ that converges in $\overline{X\times X}^{\max}$. If the limit point of this subsequence is in $X\times X \cup \Omega$ we are done. If not, either the limit is in $X\times X \cup \Delta$ or $|\{d(x'_n,o)-d(y'_n,o)\}_n|$ is unbounded for all $n$. For the first case, observe that for every $n$, $(x'_n|y'_n)_o\leq d_{X}(m_n, o)<C$, but $\{x'_n\}_n$ and $\{y'_n\}_n$ have the same limit if and only if $(x'_n|y'_n)_o\rightarrow \infty$ \cite[Part III, Sect. H.3]{BH}, hence, this case is not possible.  The second case is not possible either since $|\{d(x'_n,o)-d(y'_n,o)\}_n|$ unbounded implies $\{d(m_n, o)\}_n$ unbounded.

Now, consider a sequence of the form $\{g_n\}_n$ in $\tilde{\rho}^{-1}(K_{\Delta})$. The geodesics in the sequence satisfy $g_n(0)\in K$ for all $n$. Since the set of geodesics that go through a compact set is compact, $\{g_n\}$ has a convergent subsequence. 
\end{proof}
\begin{theorem}
\label{quotient compact}
The action of $\Gamma$ on $X\times X\cup\Omega$ is properly discontinuous and cocompact.
\end{theorem}
\begin{proof}
In order to see that $\Gamma$ acts properly discontinuously, take $K\subset (X\times X\cup\Omega)$ any compact subset, and let $\gamma\in \Gamma$ be such that $\gamma K \cap K \neq \emptyset$. Then, $$\tilde{\rho}(K)\cap \tilde{\rho}(\gamma K) = \tilde{\rho}(K)\cap \gamma \tilde{\rho}(K)\neq \emptyset.$$ Since $\tilde{\rho}$ is continuous, $\tilde{\rho}(K)$ is compact, and since the action of $\Gamma$ on $K_{\Delta}$ is properly discontinuous,  $\tilde{\rho}(K)\cap \gamma \tilde{\rho}(K)\neq \emptyset$
only for a finite number of elements $\gamma\in\Gamma$. Therefore $\gamma K \cap K \neq \emptyset$ only for a finite number of $\gamma\in\Gamma$.

For the cocompactness, observe that by Lemma \ref{fibres} and Corollary \ref{boundary of a fibre} $(X\times X\cup \Omega) /\Gamma \cong (X\times X\cup G) /\Gamma = \tilde{\rho}^{-1}(K_{\Delta})/\Gamma$ which is compact by Lemma \ref{compact domain}.
\end{proof}

\end{document}